\documentclass{article}

\usepackage{latexsym,amstext,amsmath,amssymb,amsthm, amsopn}
\DeclareMathOperator{\diver}{div}

\def \RR {\mathbb{R}}
\def \NN {\mathbb{N}}

\def \tp{\tilde{p}}
\let \u \underline

\newtheorem{thm}{Theorem} \newtheorem{lem}{Lemma}
\newtheorem{df}[lem]{Definition} 
\newtheorem{cor}[lem]{Corollary}

\title{Fractional Laplacian with singular drift\footnotetext{The research was partially supported by ANR-09-BLAN-0084-01 and grants MNiSW N N201 397137 and N N201 422539. Fellowship co-financed by European Union within European Social Fund.}}
\author{Tomasz Jakubowski \\
\small
Institute of Mathematics and Computer Science\\ \small Wroc{\l}aw University of Technology\\ 
\small Wybrze\.ze Wyspia\'nskiego 27, 50-370 Wroc{\l}aw, Poland\\
\small e-mail: Tomasz.Jakubowski@pwr.wroc.pl }
\date{}
\begin{document}
\maketitle
\begin{abstract}
For $\alpha \in (1,2)$ we consider the equation $\partial_t u =  \Delta^{\alpha/2} u - r b \cdot \nabla u$, where $b$ is a divergence free singular vector field not necessarily belonging to the Kato class. We show that for sufficiently small $r>0$ the fundamental solution is globally in time comparable with the density of the isotropic stable process.
\end{abstract}
\begin{quote}
 {\em keywords: } fractional Laplacian, gradient perturbations.
\end{quote}
\section{Introduction}

Let $d \ge 1$ be a natural number and $\alpha \in (1,2)$. We denote by $p(t,x)$ the density of the isotropic $\alpha$-stable L\'evy process, i.e.
\begin{equation}
	p(t,x) = (2\pi)^{-d} \int_{\RR^d} e^{-i x \cdot \xi} e^{-t |\xi|^{\alpha}}d\xi\,, \qquad t>0\,,\; x \in \RR^d\,.
\end{equation}

For $\varphi \in C^{\infty}_c(\RR^d)$ we define the operator
\begin{equation*}\label{p:alpha}
	\Delta^{\alpha/2} \phi(x) = \mathcal{A}_{d,\alpha} \lim_{\varepsilon \to 0^+}  \int_{|y|>\varepsilon} \frac{\phi(x+y)- \phi(x)}{|y|^{d+\alpha}}dy\,,
\end{equation*}
where $\mathcal{A}_{d,\alpha} >0$ is a constant depending only on $\alpha$ and $d$. $\Delta^{\alpha/2}$ is the infinitesimal generator of the isotropic $\alpha$-stable process with the time and space homogeneous transition density $p(t,x,y)= p(t,y-x)$, i.e.
$$
\Delta^{\alpha/2} \phi(x) = \lim_{t \to 0} \frac{1}{t} \int_{\RR^d} p(t,x,y)(\phi(y)-\phi(x))\,dy\,.
$$

Let $b(x)=(b_1(x),\ldots, b_d(x))$ be a vector field satisfying the following conditions
\begin{align}
&\sup_{x \in \RR^d}\int_{\RR^d} p(t,x,y) |b(y)| \, dy < C_b t^{1/\alpha -1},  \quad t \in (0,\infty), \label{cond:|b|}\\
&\diver b =0, \; (\mbox{\rm in the sense of distribution theory}), \label{cond:b}
\end{align}
for some constant $C_b>0$. We will study the equation
\begin{equation}\label{eq:main}
\partial_t u - \Delta^{\alpha/2} u + r b \cdot \nabla u = 0\,, \qquad x \in \RR^d,\, t>0\,.
\end{equation}
The condition (\ref{cond:|b|}) allows for functions $b$ not belonging to the usual Kato class $\mathcal{K}_d^{\alpha-1}$, see (\ref{eq:Kato} below). For example our results apply to $d=2$ and 
$$
b(y) = (y_2 |y|^{-\alpha}, -y_1 |y|^{-\alpha}).
$$
We note that $\diver b = 0$ in the sense of distributions,  $|b(y)| = |y|^{1-\alpha}$ and $b \not\in \mathcal{K}_d^{\alpha-1}$.
The main result of the paper is the following.

\begin{thm}\label{thm:main1}
There is a constant $R= R(\alpha, d, C_b)$ such that for all $ 0 \le r \le R$ there exists a function $\tp(t,x,y)$ such that
 for $\phi \in C_c^\infty(\RR,\RR^d)$, $s \in \RR$ and $x \in \RR^d$,
\begin{equation}\label{eq}
\int_s^\infty \int_{\RR^d} \tilde{p}(u-s,x,z) (\partial_u\phi(u,z) + \Delta_{z}^{\alpha/2}\phi(u,z) +r b(z) \cdot  \nabla_z \phi(u,z))dz du = -\phi(s,x)\,,
\end{equation}
and there is a constant $C>0$ depending only on $d,\alpha,b,R$ such that
\begin{equation}\label{eq:tp_est}
C^{-1} p(t,x,y) \le \tp(t,x,y) \le C p(t,x,y), \qquad t>0,\; x,y \in \RR^d\,.
\end{equation}
\end{thm}
\noindent 
According to (\ref{eq}), $\tilde{p}$ is 
the integral kernel of the left inverse of $-\big(\partial_t+\Delta^{\alpha/2}_z+ r b \cdot \nabla_z\big)$.
Put differently, since $\diver b = 0$, a function $f\colon (u,z)\mapsto \tilde{p}(u-s,x,z)$ solves 
$(\partial_t-\Delta^{\alpha/2}_z+ rb\cdot \nabla_z) f=\delta_{(s,x)}$ in the sense
of distributions. Thus, $\tilde{p}$ is the fundamental solution of (\ref{eq:main}). As a corollary we obtain that $\tp$ is the integral kernel of the Markov semigroup with the (weak) generator $\Delta^{\alpha/2} + b \cdot \nabla$ (Corollary \ref{cor:gen}). 

Equations similar to  (\ref{eq:main}) were widely studied for the Laplacian and more general elliptic operators (see e.g. \cite{1968-Aronson-asnsp}, \cite{Zhang1}, \cite{Zhang2}, \cite{2000-VL-YS-jlms}). The authors considered also drifts $b$ depending on time and satisfying various conditions similar to (\ref{eq:Kato}) and called the generalized Kato conditions . In general, the comparability of the fundamental solution with the Gaussian density holds only locally in time (i.e. with constants depending on time). To obtain global estimates, the additional assumption on divergence of $b$ is necessary. For example in \cite{1987-Osada-jmku} Osada proved that the fundamental solution of $\partial_t u =  A u + b \cdot \nabla u$  has upper and lower Gaussian bounds, where $A$ is a uniformly elliptic operator in divergence form, $b$ is the derivative of a bounded function and $\diver b=0$. These results were later obtained in \cite{2004-VL-QiZ-mm} for much more singular drifts (exceeding generalized Kato class) with some smallness assumption on divergence.

Additive perturbations of the fractional Laplacian were intensely studied in recent years (e.g. \cite{2007-TJ-spa}, \cite{2008-TJ-mz}, \cite{2007-KB-TJ-cmp}, \cite{2010-TJ-KS-jee}, \cite{2010-KB-TJ-preprint}, \cite{2010-ZQC-PK-RS-preprint}, \cite{2008-LB-GK-jee}, \cite{2009-NA-CI-tams}, \cite{2006-JD-CI-arma}).
In particular, the equation (\ref{eq:main}) was considered in \cite{2007-KB-TJ-cmp} for $b \in \mathcal{K}_d^{\alpha-1}$ but with no condition on divergence. Recall that $b \in \mathcal{K}_d^{\alpha-1}$ if
\begin{equation}\label{eq:Kato}
\lim_{t\to 0} \sup_{x \in \RR^d} \int_0^t \int_{\RR^d} s^{-1/\alpha} p(s,x,y) |b(y)|\,dy\,ds = 0
\end{equation}
(see also \cite{2011-TJ-KS-preprint} and \cite{2010-TJ-KS-jee}  for further developments). The authors obtained local in time comparability of $\tilde{p}$ and $p$ for each $r \in \RR$. The function $\tilde{p}$ was constructed as the perturbation series $\tp = \sum_{n=0}^\infty p_n$, where 
\begin{align}
p_0(t,x,y) &= p(t,x,y) \label{def:p_1},\\
p_n(t,x,y) &= \int_0^t \int_{\RR^d} p_{n-1}(t-s,x,z)b(z) \cdot \nabla_z p(s,z,y)dzds. \label{def:p_n}
\end{align}
The Kato conditions on $b$ assures local in time smallness of $p_1$ with respect to $p$, and in consequence the perturbation series converges. Similar methods was used to study Schr\"odinger perturbations of transition densities (\cite{2008-KB-WH-TJ-sm}, \cite{2009-TJ-pa}) and the Green function of $\Delta^{\alpha/2} + b(x)\cdot\nabla_x$  (\cite{2010-KB-TJ-preprint}).

In the present paper we will also use the technique of perturbation series, but in our case the conditions on $b$ only assures  the finiteness of $p_1$ (see Lemma \ref{lem:|p|_1}). In addition, the integral in (\ref{def:p_n}) is not absolutely convergent as the integral over time and space, which makes the proofs much more complicated and delicate. In particular, it is not obvious that the functions $p_n$ are well defined. In order to prove it we introduce functions $P_n(t,x,y,\u{s},\u{z})$ (see (\ref{def:p_n_ext}) for definition) and integrate it separately over $(\RR^d)^n$ and $n$-dimensional simplex $S_n(0,t)$.  Namely, we consider 
$$
|p|_n(t,x,y) = \int_{S_n(0,t)}\left|\int_{(\RR^d)^n} P_n(t,x,y,\u{s},\u{z}) d\u{z} \right| d\u{s},
$$
the majorants of the functions $p_n$ (see (\ref{def:|p|_n}) and (\ref{eq:p_n})). To estimate $|p|_n$ for $n \ge 2$ we split the integral over the simplex $S_n(0,t)$ into suitable $n+1$ parts. As a consequence Motzkin numbers appear in the estimates of $p_n$ (see \cite{2009-TJ-pa} for another connection of the perturbation series with combinatorics). In order to assure the convergence of the perturbation series, smallness of $p_1$ is needed and this is why we multiply $b$ by small constant in Theorem \ref{thm:main1}. We like to note that Theorem \ref{thm:main1} should hold for $R=\infty$, but such an extension calls for different methods.

One of the tool used in this paper is so called 3P theorem (see \cite{2007-KB-TJ-cmp}, \cite{2011-TJ-KS-preprint}, \cite{2010-TJ-KS-jee}). It allows to suitably split a ratio of three functions $p$, and in consequence to estimate $p_n$. Since for $\alpha=2$ (Gaussian case) 3P theorem does not hold, our method cannot be applied to perturbations of the classical Laplacian. Similarly as in previous papers we also exclude the case of $\alpha\le1$. Although Lemma \ref{lem:|p|_1} holds in this case, Lemma \ref{lem:p2_est} does not seem to extend to $\alpha \le 1$ and consequently our approach does not work for this case.

The paper is organized as follows. In Section 2 we collect basic properties of the transition density $p(t,x,y)$. In Section 3 we define and estimate functions $p_n$. In Section 4 we prove Theorem \ref{thm:main1}.  

All the functions considered in the sequel are Borel measurable. When
we write $f(x)\approx g(x)$, we mean that there is a number
$0<C<\infty$ independent of $x$, i.e. a {\it constant}, such that for every $x$ we have 
$C^{-1}f(x)\le g(x)\leq Cf(x)$.  As usual we write $a \land b = \min(a,b)$ and $a \vee b = \max(a,b)$.
The notation $C=C(a,b,\ldots,c)$ means that $C$ is a constant which depends
{\it only} on $a,b,\ldots,c$.

\section{Preliminaries}
Throughout the paper $d \ge 1$, and unless stated otherwise, $\alpha \in (1,2)$. In Lemmas \ref{lem:p_est}, \ref{lem:3P}, \ref{lem:grad_p} we recall well-known results about the density $p(t,x,y)$ of the isotropic $d$-dimensional $\alpha$-stable process (see \cite{2007-KB-TJ-cmp} for details). 
\begin{lem}\label{lem:p_est}
There exist a constant $C_1$ such that
\begin{equation}\label{eq:p_est}
C_1^{-1}  \left[t^{-d/\alpha} \land \frac{t}{|x|^{d+\alpha}}\right] \le p (t,x) \le C_1  \left[t^{-d/\alpha} \land \frac{t}{|x|^{d+\alpha}}\right], \qquad t \in (0,\infty), \; x \in \RR^d. 
\end{equation}
\end{lem}
\begin{lem}[3P]\label{lem:3P}
There exist a constant $C_2$ such that
$$
p(t,x,z) p(s,z,y) \le C_2 p(t+s,x,y)[p(t,x,z) + p(s,z,y)], \quad s,t \in (0,\infty),\; x,y,z \in \RR^d
$$
\end{lem}
Let $p^{(m)}$ be the $\alpha$-stable density in dimension $m$.
\begin{lem}\label{lem:grad_p}
 For all $t>0$ and $x \in \RR^d$,
\begin{equation}\label{eq:grad_p}
\nabla_x p^{(d)} (t,x) = -2\pi x p^{(d+2)}(t,\tilde{x})\,,
\end{equation}
where $\tilde{x} \in \RR^{d+2}$ is such that $|\tilde{x}| = |x|$.
\end{lem}

\noindent Applying (\ref{eq:p_est}) to (\ref{eq:grad_p}) we get
\begin{equation}\label{eq:grad_p_est}
|\nabla_x p (t,x)| \le C_3 t^{-1/\alpha} p(t,x)\,, \qquad \qquad t \in (0,\infty), \; x \in \RR^d
\end{equation}

Our aim is to prove that functions $p_n$ defined in (\ref{def:p_1}) and (\ref{def:p_n}) satisfy $p_n(t,x,y) \le C_n p(t,x,y)$, where $C_n$ are the constants such that $\sum_{n=0}^\infty C_n < \infty$. Since the condition (\ref{cond:|b|}) does not guarantee convergence of the integral 
\begin{equation}\label{eq:notconv}
\int_0^t \int_{\RR^d} p(t-s,x,z) |b(z)||\nabla_z p(s,z,y)|\,dz\,ds\,,
\end{equation}
we cannot follow the proofs from \cite{2007-KB-TJ-cmp} and \cite{2010-TJ-KS-jee}. However the inner integral of (\ref{eq:notconv}) is convergent. Indeed by (\ref{eq:grad_p_est}), Lemma \ref{lem:3P} and (\ref{cond:|b|}) we have
\begin{align}
& \int_{\RR^d} p(t-s,x,z) |b(z)||\nabla_z p(s,z,y)|\,dz \notag\\
 &\le c_1 s^{-1/\alpha} \int_{\RR^d} p(t-s,x,z) p(s,z,y) |b(z)| \,dz \notag\\
&\le c_2 s^{-1/\alpha} p(t,x,y) \int_{\RR^d} (p(t-s,x,z) + p(s,z,y)) |b(z)| \,dz \notag\\
& \le c_2C_b s^{-1/\alpha}  [(t-s)^{1/\alpha-1} + s^{1/\alpha-1}] p(t,x,y) \label{eq:inner1}
\end{align}
Therefore instead of (\ref{eq:notconv}) we will consider
\begin{equation}\label{eq:conv}
\int_0^t \left|\int_{\RR^d} p(t-s,x,z) b(z) \cdot \nabla_z p(s,z,y)\,dz\right|\,ds\,.
\end{equation}
In order to estimate (\ref{eq:conv}) we will use the following lemma.
\begin{lem}\label{lem:Move_grad}
For all $s,t \ge 0$ and $x,y \in \RR^d$ we have
\begin{align}\label{lem:move_grad}
\int_{\RR^d} p(t,x,z) b(z) \cdot \nabla_z p(s,z,y)\, dz = - \int_{\RR^d} \nabla_z p(t,x,z) \cdot b(z) p(s,z,y)\, dz\,.
\end{align}
\end{lem}
\begin{proof}
Let $g \in C_c^\infty(\RR^d)$ be such that 
$$
g(z) =
\begin{cases}
1 &{\rm for \;} |z| \le 1,\\
0 &{\rm for \;} |z| \ge 2\,.
\end{cases}
$$
Let $f_n(z) = g(z/n) p(t,x,z) p(s,z,y) \in C_c^\infty(\RR^d)$. Then $\nabla_z f_n(z) \to \nabla_z(p(t,x,z) p(s,z,y))$ as $n \to \infty$. Furthermore 
$$
|\nabla_z f_n(z)| \le c(|\nabla_z(p(t,x,z) p(s,z,y))| + p(t,x,z) p(s,z,y) )
$$
for some constant $c>0$.
By (\ref{eq:grad_p_est}), Lemma \ref{lem:3P} and (\ref{cond:|b|})
$$
\int_{\RR^d} [|\nabla_z(p(t,x,z) p(s,z,y))| + p(t,x,z) p(s,z,y) ] |b(z)|\,dz < \infty\,.
$$
Therefore by (\ref{cond:b}) and Lebesgue theorem
$$
0 = \lim_{n \to \infty}\int_{\RR^d} \nabla_z f_n(z) \cdot b(z)\, dz = \int_{\RR^d}\nabla_z(p(t,x,z) p(s,z,y)) \cdot b(z)\, dz\,,
$$
which ends the proof.
\end{proof}
Similarly condition (\ref{cond:b}) yields that for all $s,t>0$ and $\xi, y \in \RR^d$
\begin{align}
\begin{split}\label{eq:move_grad}
&\int_{\RR^d} [b(\xi) \cdot \nabla_\xi p(t,\xi,z)] b(z) \cdot \nabla_z p(s,z,y)\, dz \\
&= - \int_{\RR^d} \nabla_z [b(\xi) \cdot \nabla_\xi p(t,\xi,z)] \cdot b(z) p(s,z,y)\, dz\,.
\end{split}
\end{align}
In the following lemma we will use  (\ref{lem:move_grad}) to show that the function $p_1$ introduced in (\ref{def:p_n}) is well defined. In a similar way we will apply (\ref{eq:move_grad}) to estimate other functions $p_n$. 
\begin{lem}\label{lem:|p|_1}
There exists a constant $C$ such that for all $t>0$, $x,y \in \RR^d$, 
\begin{equation}
\int_0^{t} \left| \int_{\RR^d} p(t-s,x,z) b(z) \cdot \nabla_z p(s,z,y)\, dz\right|\, ds  \le C p(t,x,y)\,.
\end{equation}
\end{lem}
\begin{proof}
By Lemma \ref{lem:Move_grad} and (\ref{eq:inner1}) we obtain
\begin{align*}
&\int_0^{t} \left| \int_{\RR^d} p(t-s,x,z) b(z) \cdot \nabla_z p(s,z,y)\, dz\right|\, ds \notag\\ 
& =  \int_{0}^{t/2} \left| \int_{\RR^d} p(t-s,x,z) b(z) \cdot \nabla_z p(s,z,y)\, dz\right|\, ds \notag \\
& + \int_{t/2}^t \left| \int_{\RR^d} p(t-s,x,z) b(z) \cdot \nabla_z p(s,z,y)\, dz\right|\, ds \notag \\ 
& =  \int_0^{t/2} \left| \int_{\RR^d}   \nabla_zp(t-s,x,z)\cdot b(z)   p(s,z,y)\, dz\right|\, ds \notag\\
& +  \int_{t/2}^t \left| \int_{\RR^d} p(t-s,x,z) b(z) \cdot \nabla_z p(s,z,y)\, dz\right|\, ds \le cp(t,x,y) \notag 
\end{align*}
\end{proof}
\noindent We will also need the following two auxiliary lemmas
\begin{lem}\label{lem:aux1}
There exists a constant $C_4$ such that for all $t>0$ and $z,w \in \RR^d$ we have
$$
|b(z) \cdot\nabla_z (b(w) \cdot\nabla_w p(t,z,w))| \le C_4 |b(z)| |b(w)| p(t,z,w) t^{-2/\alpha}\,.
$$
\end{lem}
\begin{proof}
From (\ref{eq:grad_p}) we get
\begin{align*}
\nabla_w p^{(d)} (t,z,w) & =  2\pi (z-w) p^{(d+2)}(t,\tilde{z},\tilde{w}),\\
\nabla_z p^{(d+2)} (t,\tilde{z},\tilde{w}) & = -2\pi (z-w) p^{(d+4)}(t,\hat{z},\hat{w}),
\end{align*}
where $\tilde{z} = (z,0,0) \in \RR^{d+2}$ and $\hat{z} = (\tilde{z},0,0) \in \RR^{d+4}$ (we use similar notation for $\tilde{w}$ and $\hat{w}$).
Therefore,
\begin{align*}
& b(z) \cdot\nabla_z (b(w) \cdot\nabla_w p(t,z,w)) = 2\pi b(z) \cdot \nabla_z [b(w) \cdot(z-w) p^{(d+2)}(t,\tilde{z},\tilde{w})] \\
& =2\pi b(z) \cdot [b(w) p^{(d+2)}(t,\tilde{z},\tilde{w}) - 2\pi (b(w) \cdot(z-w)) (z-w)p^{(d+4)}(t,\hat{z},\hat{w})] \\
& =2\pi b(z) \cdot b(w) p^{(d+2)}(t,\tilde{z},\tilde{w}) - 4\pi^2 (b(w) \cdot(z-w)) b(z) \cdot (z-w) p^{(d+4)}(t,\hat{z},\hat{w})
\end{align*}
Applying (\ref{eq:p_est}) we obtain the assertion of the lemma.
\end{proof}

\begin{lem}\label{lem:aux2}
There exists a constant $C_5$ such that for all $t>0$
$$
\int_0^{t/2} \int_{t/2}^t (r-u)^{-\frac{2}{\alpha}} ((t-r)^{\frac{1}{\alpha}-1} + (r-u)^{\frac{1}{\alpha}-1}) ((t-u)^{\frac{1}{\alpha}-1} + u^{\frac{1}{\alpha}-1})\,dr\,du < C_5\,.
$$
\end{lem}

\begin{proof}
First we note that $a^p+b^p \le 2^{1-p} (a + b)^p$ for $a,b \ge 0$ and $0<p<1$. Consequently  $a^{-p} + b^{-p} \le 2^{1-p} (a + b)^p (ab)^{-p}$. Hence it suffices to show
$$
\int_0^{t/2} \int_{t/2}^t (r-u)^{-\frac{1}{\alpha}-1} (t-r)^{\frac{1}{\alpha}-1}  u^{\frac{1}{\alpha}-1} t^{1-\frac{1}{\alpha}}  \,dr\,du < c.
$$ 
Splitting the second integral into intervals $(t/2,3t/4)$ and $(3t/4,t)$ we get
$$
 t^{1-\frac{1}{\alpha}} \int_0^{t/2} \int_{t/2}^t (r-u)^{-\frac{1}{\alpha}-1} (t-r)^{\frac{1}{\alpha}-1}  u^{\frac{1}{\alpha}-1}\,dr\,du 
\le c_1   \int_0^{t/2} (t/2-u)^{-\frac{1}{\alpha}} u^{\frac{1}{\alpha}-1} < c_2\,.
$$ 
\end{proof}

We note that Lemma \ref{lem:|p|_1} extends to $\alpha \le 1$, however the following lemma does not and this is why we generally assume $\alpha \in (1,2)$ in the paper. Lemma \ref{lem:p2_est} will allow us to estimate the functions $p_n$ for $n \ge 2$.

\begin{lem}\label{lem:p2_est}
There exists a constant $C$ such that for all $t >0$ and $x,y \in \RR^d$,
\begin{align}
&\int_0^{t/2}\int_{t/2}^t \int_{\RR^d}\int_{\RR^d}\,dw\,d\xi\,dr\,du  \nonumber\\
& p(u,x,\xi) \big|\nabla_w\big(b(\xi) \cdot \nabla_\xi p(r-u,\xi,w) \big) \cdot b(w)\big| p(t-r,w,y)  < Cp(t,x,y)\,. \label{Eq:p2_est}
\end{align}
\end{lem}
\begin{proof}
By Lemma \ref{lem:3P} and (\ref{cond:|b|})
\begin{align*}
\int_{\RR^d} p(s,x,z) |b(z)| p(t,z,y)\,dz \le c p(s+t,x,y) (s^{1/\alpha-1} + t^{1/\alpha-1})\,.
\end{align*}
Hence by Lemma \ref{lem:aux1} we have
\begin{align*}
&\int_{\RR^d}\int_{\RR^d} p(u,x,\xi) \big|\nabla_w\big(b(\xi) \cdot \nabla_\xi p(r-u,\xi,w) \big) \cdot b(w)\big| p(t-r,w,y) \,dw\,d\xi\\
& \le \int_{\RR^d}\int_{\RR^d}   p(u,x,\xi) |b(\xi)| (r-u)^{-2/\alpha} p(r-u,\xi,w) |b(w)| p(t-r,w,y) \,dw\,d\xi \\
& \le \int_{\RR^d}   |b(\xi)| (r-u)^{-2/\alpha} p(u,x,\xi)p(t-u,\xi,y) ((t-r)^{1/\alpha-1} + (r-u)^{1/\alpha-1}) \,dw \\
& \le p(t,x,y) (r-u)^{-2/\alpha}((t-r)^{1/\alpha-1} + (r-u)^{1/\alpha-1}) ((t-u)^{1/\alpha-1}+u^{1/\alpha-1})\,.
\end{align*}
Now (\ref{Eq:p2_est}) follows from Lemma \ref{lem:aux2}.
\end{proof}

\section{Perturbation series}

In this section we introduce functions $|p|_n$ which will be majorants of the functions $p_n$ (see (\ref{def:p_n})). For any $a <b$ and $n \ge 1$ we denote
$$
S_n(a,b) = \{(s_1,s_2,\ldots,s_n) \in \RR^n \colon a \le s_1 \le s_2 \le \ldots \le s_n \le b\}\,.
$$
For any $t>0$ and $x,y \in \RR^d$ let
\begin{equation}\label{def:p_n_ext}
P_n(t,x,y,\u{s},\u{z}) = p(s_1,x,z_1)b(z_1) \cdot \nabla_{z_1} p(s_2-s_1,z_1,z_2) \ldots b(z_n) \cdot \nabla_{z_n}p(t-s_n,z_n,y)\,,
\end{equation}
where $\u{s}=(s_1,\ldots,s_n) \in S_n(0,t)$ and $\u{z} = (z_1,\ldots,z_n) \in (\RR^d)^n$\,.
We recall that the integrals in (\ref{def:p_n}) may not be absolutely convergent as the integrals over $[0,t] \times \RR^d$, and the approach from the paper \cite{2007-KB-TJ-cmp} cannot be used. Therefore we separate integrals over time and space in the following definition.
\begin{df}
For any $t>0$ and $x,y \in \RR^d$ we define
\begin{align}
|p|_0(t,x,y) &= p(t,x,y) \label{def:|p|_0} \\
|p|_n(t,x,y) &=\int_{S_n(0,t)} \left|\int_{(\RR^d)^n} P_n(t,x,y,\u{s},\u{z}) d\u{z} \right| d\u{s}.  \label{def:|p|_n}
\end{align}
\end{df}

\noindent We note that by Lemma \ref{lem:3P}, (\ref{cond:|b|}) and (\ref{eq:grad_p_est}) 
\begin{equation}\label{eq:P_n}
\int_{(\RR^d)^n} |P_n(t,x,y,\u{s},\u{z})| d\u{z} < \infty,
\end{equation}
hence the functions $|p|_n$ are well-defined (at most they are equal to infinity). However the integral $\int_{S_n(0,t)} \int_{(\RR^d)^n} \left|P_n(t,x,y,\u{s},\u{z})\right| d\u{z}  d\u{s}$ may not be convergent because singularities of the gradient of the functions $p$ in (\ref{def:p_n_ext}) may be not integrable in the whole simplex $S_n(0,t)$. Therefore in order to estimate (\ref{def:|p|_n}) we use the following representation 
\begin{align}
S_n(0,t) = S_n(0,t/2) \cup \left(\bigcup_{k=1}^{n-1} S_{n-k}(0,t/2) \times S_k(t/2,t) \right) \cup S_n(t/2,t) \label{eq:S}\,,
\end{align}
Lemma \ref{lem:Move_grad} and (\ref{eq:move_grad}) to move these singularities off the region of integration.  
\begin{lem}\label{lem:|p|_m_est}
For any $1 \le k \le n-1$, $t>0$, and $x,y \in \RR^d$ we have
\begin{align}
&\int_{S_{n-k}(0,t/2)\times S_k(t/2,t)} \bigg| \int_{(\RR^d)^n} P_n(t,x,y,\u{s},\u{z}) \,d\u{z}\bigg|\,d\u{s} \nonumber\\
\le &\int_0^{t/2}\int_{t/2}^t \int_{\RR^d}\int_{\RR^d} \,dw\,d\xi \label{eq:|p|_m_est}\\
& |p|_{n-k-1}(u,x,\xi) \big|\nabla_w\big(b(\xi) \cdot \nabla_\xi p(r-u,\xi,w) \big) \cdot b(w)\big| |p|_{k-1}(t-r,w,y) \,dr\,du\,. \nonumber
\end{align} 
\end{lem}

\begin{proof}
By (\ref{eq:P_n}) and Fubini's theorem we may change the order of integration in integrals over $(\RR^d)^n$. We note that
\begin{equation}\label{eq1:|p|_m_est}
|p|_m(t-r,x,y) = \int_{S_m(r,t)} \left|\int_{(\RR^d)^m} P_m(t-r,x,y,\u{s},\u{z})\,d\u{z}\right|\,d\u{s}\,.
\end{equation}
Changing the order of variables in the following way,  
$$ 
\int_{S_m(a,b)} f(\u{s}) ds_m ds_{m-1} \ldots ds_1 = \int_{(a,b) \times S_{m-1}(a,s_m)} f(\u{s}) ds_{m-1} \ldots ds_1 ds_m\,,
$$
using (\ref{eq:move_grad}) and Fubini's theorem we get
\begin{align*}
&\int_{S_{n-k}(0,t/2)\times S_k(t/2,t)} \bigg| \int_{(\RR^d)^n} P_n(t,x,y,\u{s},\u{\xi}) \,d\u{\xi}\bigg|\,d\u{s} \\
& = \int_0^{t/2}\int_{ S_{n-k-1}(0,r_{n-k})} \int_{t/2}^t \int_{S_{k-1}(u_k,t)} d\u{u}\, du_k\, d\u{r}\,dr_{n-k}\\ 
&  \bigg| \int_{(\RR^d)^n} P_{n-k-1}(r_{n-k},x,z_{n-k},\u{r},\u{z}) b(z_{n-k})\cdot \nabla_{z_{n-k}} p(u_1-r_{n-k},z_{n-k}, w_1)\\
&   b(w_1)\cdot\nabla_{w_1} P_{k-1}(t-u_1,w_1,y,\u{u},\u{w}) \,d\u{w} \,dw_1\, d\u{z}\,dz_{n-k} \bigg| \\
& = \int_0^{t/2}\int_{ S_{n-k-1}(0,r_{n-k})} \int_{t/2}^t \int_{S_{k-1}(u_k,t)} d\u{u}\, du_k\, d\u{r}\,dr_{n-k}\\ 
&  \bigg| \int_{(\RR^d)^n} P_{n-k-1}(r_{n-k},x,z_{n-k},\u{r},\u{z}) P_{k-1}(t-u_1,w_1,y,\u{u},\u{w}) \\
&   b(w_1)\cdot\nabla_{w_1} [b(z_{n-k})\cdot \nabla_{z_{n-k}} p(u_1-r_{n-k},z_{n-k}, w_1)] \,d\u{w} \,dw_1\, d\u{z}\,dz_{n-k}  \bigg|\,,
\end{align*}
where $\u{r}= (r_1,\dots,r_{n-k-1})$, $\u{u}= (u_1,\dots,u_{k-1})$, $\u{z}= (z_1,\dots,z_{n-k-1})$, $\u{w}= (w_2,\dots,w_k)$. Now splitting integral over $(\RR^d)^n$ into integrals over $(\RR^d)^{n-k-1}$, $(\RR^d)^{k-1}$, and $(\RR^d)^2$ and applying (\ref{eq1:|p|_m_est}), we get (\ref{eq:|p|_m_est})
\end{proof}

\begin{lem}\label{lem:|p|_well_def}
For $n \ge 2$, $t>0$ and $x,y \in \RR^d$ we have 
\begin{align*}
|p|_n(t,x,y) &\le \int_0^{t/2} \int_{\RR^d} |p|_{n-1}(u,x,z) |b(z)\cdot \nabla_z p(t-u,z,y)|\,dz \,du\\
& + \int_{t/2}^t \int_{\RR^d} |\nabla_z p(u,x,z)\cdot b(z)| |p|_{n-1}(t-u,z,y)\,dz \,du \\
&+ \sum_{k=0}^{n-2} \int_0^{t/2} \int_{t/2}^t \int_{\RR^d}\int_{\RR^d} \,dw\,dz\,dr\,du |p|_k(u,x,z) \times \\
& \hskip1cm \times \big|b(z) \cdot\nabla_z (b(w) \cdot\nabla_w p(r-u,z,w))\big| |p|_{n-2-k}(t-r,w,y)\,.  
\end{align*}
\end{lem}
\begin{proof}
By (\ref{eq:S}) we get 
\begin{align}
|p|_n(t,x,y) &= \int_{S_n(0,t/2)} \bigg| \int_{(\RR^d)^n} P_n(t,x,y,\u{s},\u{\xi}) \,d\u{\xi}\bigg|\,d\u{s}  \label{eq1:|p|_well_def}\\
& + \int_{S_n(t/2,t)} \bigg| \int_{(\RR^d)^n} P_n(t,x,y,\u{s},\u{\xi}) \,d\u{\xi}\bigg|\,d\u{s} \label{eq3:|p|_well_def} \\
&+ \int_{\bigcup_{k=1}^{n-1} S_{n-k}(0,t/2) \times S_k(t/2,t)  } \bigg| \int_{(\RR^d)^n} P_n(t,x,y,\u{s},\u{\xi}) \,d\u{\xi}\bigg|\,d\u{s} \label{eq2:|p|_well_def}
\end{align}
The integral (\ref{eq1:|p|_well_def}) is estimated as follows,
\begin{align*}
& \int_{S_n(0,t/2)} \bigg| \int_{(\RR^d)^n} P_n(t,x,y,\u{s},\u{\xi}) \,d\u{\xi}\bigg|\,d\u{s} \\
& = \int_0^{t/2} \int_{S_{n-1}(0,s_n)} \bigg| \int_{(\RR^d)^n} P_{n-1}(s_n,x,\xi_n,\u{s}^*,\u{\xi}^*) b(\xi_n)\cdot\nabla_{\xi_n} p(t-s_n,\xi_n,y) \,d\u{\xi}\bigg|\,d\u{s} \\
& \le \int_0^{t/2} \int_{S_{n-1}(0,s_n)} \int_{\RR^d} \,d\xi_n \,d\u{s}^*\,ds_n  \times \\
& \times  |b(\xi_n)\nabla_{\xi_n} p(t-s_n,\xi_n,y)| \bigg|\int_{(\RR^d)^{n-1}} P_{n-1}(s_n,x,\xi_n,\u{s}^*,\u{\xi}^*)  \,d\u{\xi}^*\bigg| \\
& \le  \int_0^{t/2}  \int_{\RR^d} |p|_{n-1}(s_n,x,\xi_n) |b(\xi_n) \cdot \nabla_{\xi_n} p(t-s_n,\xi_n,y)| \,d\xi_n\,ds_n\,,
\end{align*}
where $\u{s}^* = (s_1,\ldots s_{n-1})$ and $\u{\xi}^* = (\xi_1,\ldots \xi_{n-1})$. 
Applying Lemma \ref{lem:Move_grad} and using similar method we estimate (\ref{eq3:|p|_well_def}). Next, we split (\ref{eq2:|p|_well_def}) into $n-1$ integrals over the sets $S_{n-k}(0,t/2)\times S_k(t/2,t)$ and apply Lemma \ref{lem:|p|_m_est} to each integral.
\end{proof}

By the lemmas from the previous section and induction we will obtain that all functions $|p|_n$ are finite and in consequence the functions $p_n$ are well defined. Detailed estimates will we given in the next section. 

\section{Proof of Theorem 1}
Before we pass to the proofs of the main theorem we briefly introduce the Motzkin numbers. In combinatorics Motzkin number $M_n$ represents the number of different ways of drawing non-intersecting chords on a circle between $n$ points (\cite{1948-TM-bams}). Their generating function is (see \cite{1999-RS-EnumerCombin})
\begin{equation}\label{eq:Motz_gen}
	M(x) = \sum_{n=0}^\infty M_n x^n = \frac{1-x-\sqrt{1-2x-3x^2}}{2x^2}\,,
\end{equation}
and the following recurrence relation holds,
\begin{equation}\label{eq:Motz_rec}
M_n = M_{n-1} +\sum_{k=0}^{n-2} M_k M_{n-2-k}\,, \qquad M_0 = M_1 = 1\,.	
\end{equation}

We may now prove the main estimates of this paper
\begin{lem} \label{lem:|p|_n}
There is a constant $C$ such that for all $t>0$, $x,y \in \RR^d$ and $n \ge 1$,
\begin{equation}\label{Eq:|p|_n}
|p|_n(t,x,y) \le M_n C^n p(t,x,y)\,.
\end{equation}
\end{lem}
\begin{proof}
Let $c_1$ be the constant such that (see the proof of Lemma  \ref{lem:|p|_1})
\begin{align}
& \int_0^{t/2}  \int_{\RR^d} p(s,x,z) |b(z) \cdot \nabla_z p(t-s,z,y)|\, dz\, ds   \nonumber\\
& + \int_{t/2}^t \int_{\RR^d}   |\nabla_z p(s,x,z) \cdot  b(z)|   p(t-s,z,y)\, dz\, ds  \le c_1 p(t,x,y)\,. \label{eq:|p|_n}
\end{align}
Let $C = c_1 \vee \sqrt{c_2}$, where $c_2$ is the constant from Lemma \ref{lem:p2_est}.  

We use induction. For $n=1$ we apply Lemma \ref{lem:|p|_1}. Suppose (\ref{Eq:|p|_n}) holds for $n=1,\ldots k-1$. By Lemma \ref{lem:|p|_well_def} we get
\begin{align*}
|p|_k(t,x,y) &\le C^{k-1} M_{k-1} \int_0^{t/2} \int_{\RR^d} p(u,x,z) |b(z)\cdot \nabla_z p(t-u,z,y)|\,dz \,du\\
& + C^{k-1} M_{k-1} \int_{t/2}^t \int_{\RR^d} |\nabla_z p(u,x,z)\cdot b(z)| p(t-u,z,y)\,dz \,du  \\
&+ \sum_{j=0}^{k-2} C^j M_j C^{k-2-j} M_{k-2-j} 
 \int_0^{t/2} \int_{t/2}^t\int_{\RR^d}\int_{\RR^d} \,dw\,dz \,dr\,du \\
& p(u,x,z) \big|b(z) \cdot\nabla_z (b(w) \cdot\nabla_w p(r-u,z,w))\big|\times  
  p(t-r,w,y)\,.  
\end{align*}
Now by  (\ref{eq:|p|_n}), Lemma \ref{lem:p2_est} and (\ref{eq:Motz_rec}) we obtain
\begin{align*}
|p|_k(t,x,y) & \le \left(C^{k-1} c_1 M_{k-1} +  \sum_{j=0}^{k-2} C^j M_j C^{k-2-j} M_{k-2-j} c_2\right)p(t,x,y) \\
& \le C^k \left(M_{k-1} +\sum_{j=0}^{k-2} M_j M_{k-2-j}\right)p(t,x,y) = C^kM_k p(t,x,y)\,.
\end{align*}
\end{proof}
For all $n \in \NN$ let $p_n$ be functions satisfying (\ref{def:p_1}) and (\ref{def:p_n}).
\begin{cor}\label{cor:p_n}
Functions $p_n$ are well defined and there is a constant $C$ such that for all $t>0$, $x,y \in \RR^d$ and $n \ge 1$ 
\begin{equation}\label{Eq:p_n}
|p_n(t,x,y)| \le M_n C^n p(t,x,y)\,.
\end{equation}
\end{cor}
\begin{proof}
We simultaneously prove the estimates of $p_n$ and that they are well defined.
It suffices to show that for $n \ge 1$
\begin{equation}\label{eq:p_n}
p_n(t,x,y) = \int_{S_n(0,t)}\int_{(\RR^d)^n} P_n(t,x,y,\u{s},\u{z})\,d\u{z}\,d\u{s}, \quad t>0,\; x,y \in \RR^d\,.
\end{equation}
We use induction. For $n=1$ (\ref{eq:p_n}) matches the definition of $p_1$. Suppose (\ref{eq:p_n}) holds for $n \in \NN$. By Lemmas \ref{lem:|p|_n}, \ref{lem:3P} and \ref{lem:grad_p} be have 
\begin{align}
& \int_{\RR^d}\int_{S_n(0,t-u)}\left|\int_{(\RR^d)^n} P_n(t-u,x,\xi,\u{s},\u{z})\,d\u{z}\right| |b(\xi)| | \nabla_\xi p(u,\xi,y)| d\u{s}  d\xi \notag\\
& \le c\int_{\RR^d} p(t-u,x,\xi)  |b(\xi)| u^{-1/\alpha} p(u,\xi,y)| d\xi \notag \\
& \le c ((t-u)^{1/\alpha-1} + u^{1/\alpha-1}) u^{-1/\alpha} p(t,x,y). \label{eq1:cor:p_n}
\end{align}
Therefore by Fubini's theorem and (\ref{eq1:cor:p_n}) 
\begin{align*}
 &p_{n+1}(t,x,y) \\
= &\int_0^t \int_{\RR^d} p_n(t-u,x,\xi) b(\xi)\cdot \nabla_\xi p(u,\xi,y) d\xi du \\
= &\int_0^t \int_{\RR^d} \int_{S_n(0,t-u)}\int_{(\RR^d)^n} P_n(t-u,x,\xi,\u{s},\u{z})\,d\u{z}\,d\u{s} b(\xi)\cdot \nabla_\xi p(u,\xi,y) d\xi du\\
= &\int_0^t  \int_{S_n(0,t-u)}\int_{\RR^d}\int_{(\RR^d)^n} P_n(t-u,x,\xi,\u{s},\u{z}) b(\xi)\cdot \nabla_\xi p(u,\xi,y)\,d\u{z}\, d\xi d\u{s} du\\
\le & \, |p|_{n+1}(t,x,y)\,,
\end{align*}
which ends the proof
\end{proof}

\begin{proof}[Proof of Theorem \ref{thm:main1}]
Let $C$ be the constant from Corollary \ref{cor:p_n}. Let $r$ be such that $rC < (\sqrt{5}-1)/4$. Then 
 \begin{equation}\label{eq:1}
|p_n(t,x,y)| \le M_n (rC)^n p(t,x,y)\,.
\end{equation}
Denote $\eta =rC$. We define $\tilde{p}(t,x,y)$ as
\begin{equation}\label{def:tp}
\tilde{p}(t,x,y) = \sum_{n=0}^\infty p_n(t,x,y)
\end{equation}
By (\ref{eq:1}) and (\ref{eq:Motz_gen}), the series converges, and
$$
\tilde{p}(t,x,y) \le \frac{1-\eta-\sqrt{1-2\eta-3\eta^2}}{2\eta^2} p(t,x,y)\,.
$$
Furthermore,
$$
\tilde{p}(t,x,y) \ge p(t,x,y)  - \sum_{n=1}^\infty|p|_n(t,x,y)  \ge  \frac{4\eta^2-1 +\eta+\sqrt{1-2\eta-3\eta^2}}{2\eta^2} p(t,x,y) 
$$
We next prove that for $\phi \in C_c^\infty(\RR,\RR^d)$, $s \in \RR$ and $x \in \RR^d$,
$$
\int_s^\infty \int_{\RR^d} \tilde{p}(u-s,x,z) (\partial_u\phi(u,z) + \Delta_{z}^{\alpha/2}\phi(u,z) + b(z) \cdot  \nabla_z \phi(u,z))dz du = -\phi(s,x).
$$
By (\ref{def:tp}) we get
\begin{align}
\tp(t,x,y) &= p(t,x,y) + \sum_{n=1}^\infty p_n(t,x,y) \notag\\
&=p(t,x,y) + \sum_{n=1}^\infty \int_0^t \int_{\RR^d}p_{n-1}(t-s,x,z)b(z)\cdot\nabla_zp(s,z,y)dzds \notag\\
&= p(t,x,y) + \int_0^t \int_{\RR^d}\tp(t-s,x,z)b(z)\cdot\nabla_zp(s,z,y)dzds \label{eq:Duhamel}
\end{align}
Here Fubini theorem is justified by similar arguments as in the proof of (\ref{eq:p_n}).
The rest of the proof is the same as in \cite[Theorem 1]{2010-TJ-KS-jee}
\end{proof}
\begin{cor}\label{cor:gen}
\rm The function $\tilde{p}$ satisfies the Chapman-Kolmogorov equation 
\begin{equation}\label{eq:Ch-K}
\int_{\RR^d} \tilde{p}(s,x,z)\tilde{p}(t,z,y)\,dz = \tilde{p}(t+s,x,y), \qquad s,t>0,\; x,y \in \RR^d, 
\end{equation}
and a family of operators
$$
\tilde{P}_t f(x) = \int_{\RR^d} \tilde{p}(t,x,y) f(y)\,dy 
$$
forms a Markov semigroup with a (weak) generator $\Delta^{\alpha/2} + b(x) \cdot \nabla_x$.
\end{cor}
\begin{proof}
For the proof of (\ref{eq:Ch-K}) see  \cite[Lemmas 15, 16]{2010-TJ-KS-jee}. By (\ref{eq:Duhamel}), (\ref{eq:tp_est}) and Fubini theorem 
\begin{align*}
\int_{\RR^d} \tp(t,x,y)dy  &= \int_{\RR^d} \left(p(t,x,y) + \int_0^t \int_{\RR^d}\tp(t-s,x,z)b(z)\cdot\nabla_zp(s,z,y)dzds \right)dy \\
&= 1 + \int_0^t \int_{\RR^d}\tp(t-s,x,z)b(z)\cdot\nabla_z \left(\int_{\RR^d} p(s,z,y) dy \right) dzds = 1\,.
\end{align*}
Now, let $f,g \in C_c^\infty(\RR^d)$. We will show that 
\begin{equation}
\lim_{t\to 0} \int_{\RR^d} \frac{\tilde{P}_t f(x) - f(x)}{t}\, g(x)\,dx = \int_{\RR^d} (\Delta^{\alpha/2}f(x) + b(x)\cdot \nabla f(x)) g(x)dx
\end{equation}
By (\ref{eq:Duhamel})
\begin{align*}
&\int_{\RR^d} \frac{\tilde{P}_t f(x) - f(x)}{t}\, g(x)\,dx \\
&= \int_{\RR^d}\int_{\RR^d} \frac{p(t,x,y)( f(y) - f(x))}{t}\,g(x)\,dy \,dx \\
&+ \frac{1}{t}\int_{\RR^d} \int_{\RR^d}\int_0^t \int_{\RR^d} \tp(t-s,x,z) b(z)\cdot \nabla_z p(s,z,y) f(y) g(x) \,dz\,ds\,dy\,dx \\
&= I_1(t) + I_2(t)\,.
\end{align*}
The first summand converges to $\int_{\RR^d} \Delta^{\alpha/2}f(x) g(x)dx$. By careful use of the Fubini theorem
$$
I_2(t) = \frac{1}{t}\int_{\RR^d} \int_{\RR^d}\int_{\RR^d}\int_0^t  \tp(t-s,x,z) p(s,z,y)  b(z)\cdot \nabla_y f(y) g(x) \,dz\,ds\,dy\,dx\,.
$$
If we denote by $p^{(b)}$ the function $p$ perturbed by $b$ then $\tp(t,x,y)=p^{(b)}(t,x,y) = p^{(-b)}(t,y,x)$. Hence $\int_{\RR^d} \tp(t,x,y)dx=1$ for $t>0$ and $y\in\RR^d$. Therefore by (\ref{eq:tp_est})
\begin{align*}
\left|I_2(t) - \int_{\RR^d} b(z)\cdot\nabla f(z)\, g(z)dz \right| \le & c\int_{\RR^d} \int_{\RR^d}\int_{\RR^d}\int_0^t \frac{p(t-s,x,z) p(s,z,y)}{t} | b(z)| \times\\
& \times |\nabla_y f(y) g(x)- \nabla_z f(z) g(z)| \,dz\,ds\,dy\,dx,
\end{align*}
and the last expression converges to $0$ as $t \to 0$ (for details see the proof of \cite[Theorem1]{2007-KB-TJ-cmp}).
\end{proof}
\section*{Acknowledgment} 
I would like to thank Krzysztof Bogdan, Grzegorz Karch, Jacek Zienkiewicz and Wojciech Mydlarczyk for many helpful
comments on this paper.

\bibliographystyle{abbrv} 
\bibliography{zgk}
\end{document}